\documentclass[12pt]{amsart}
\usepackage{amssymb,latexsym,amsmath,amsthm,enumitem,mathrsfs,geometry}
\usepackage{fullpage}
\usepackage{fancyhdr}
\usepackage{xcolor}
\usepackage{bbm}
\usepackage{stmaryrd}
\usepackage{mathrsfs}
\usepackage{hyperref}
\usepackage{lineno}
\usepackage{diagbox}
\usepackage{array}
\usepackage{tikz}
\usetikzlibrary{arrows}
\usetikzlibrary{positioning}
\usepackage{float}
\usepackage{comment}
\usepackage{mathtools}
\usepackage{cleveref}

\newtheorem{theorem}{Theorem}[section]
\newtheorem*{theorem*}{Theorem}
\newtheorem{lemma}{Lemma}[section]
\newtheorem*{remark*}{Remark}

\newtheorem{definition}{Definition}[section]

\newcommand{\ayla}[1]{{\color{violet}{Ayla: #1}}}

\newcommand{\eps}{\varepsilon}
\newcommand{\R}{\mathbb{R}}

\newcommand{\todo}[1]{{\color{red}{\sc(#1)}}}

\begin{document}

\numberwithin{equation}{section}

\title{On the number of exceptional intervals to the prime number theorem in short intervals}

\author[Gafni]{Ayla Gafni}
\address{Department of Mathematics, University of Mississippi, University MS 38677}
\email{argafni@olemiss.edu}
\author[Tao]{Terence Tao}
\address{Department of Mathematics, UCLA, 405 Hilgard Ave, Los Angeles CA 90024}
\email{tao@math.ucla.edu}

\keywords{}
\subjclass[2020]{}
\thanks{}

\date{\today}

\begin{abstract} For a fixed exponent $0 < \theta \leq 1$, it is expected that we have the prime number theorem in short intervals $\sum_{x \leq n < x+x^\theta} \Lambda(n) \sim x^\theta$ as $x \to \infty$.  From the recent zero density estimates of Guth and Maynard, this result is known for all $x$ for $\theta > \frac{17}{30}$ and for almost all $x$ for $\theta > \frac{2}{15}$.  Prior to this work, Bazzanella and Perelli obtained some upper bounds on the size of the exceptional set where the prime number theorem in short intervals fails.  We give an explicit relation between zero density estimates and exceptional set bounds, allowing for the most recent zero density estimates to be directly applied to give upper bounds on the exceptional set via a small amount of computer assistance.
\end{abstract}

\maketitle

\section{Introduction and background} \label{intro}

\subsection{Prime number theorem in short intervals}

The prime number theorem asserts that\footnote{See \Cref{notation-sec} for our conventions on asymptotic notation.}
$$ \sum_{n \leq x} \Lambda(n) \sim x$$
as $x \to \infty$, where $\Lambda$ is the von Mangoldt function.  Based on this, as well as random models such as the Cram\'er random model \cite{cramer}, it is natural to conjecture that the prime number theorem also holds in short intervals in the sense that
\begin{equation}\label{xyn}
 \sum_{x < n \leq x+y} \Lambda(n) \sim y
\end{equation}
as $x \to \infty$ whenever $y = x^\theta$ for a fixed $0 < \theta \leq 1$; the classical prime number theorem then corresponds to the case $\theta=1$.  An easy averaging argument (cf.  \cite[Corollary 1(i)]{bp}) shows that the difficulty of this claim increases as $\theta$ decreases, or equivalently if \eqref{xyn} holds for $\theta = \theta_0$ then it also holds for $\theta_0 \leq \theta \leq 1$.  A well-known result of Maier \cite{Maier} shows that the bound \eqref{xyn} can break down if $y$ is only of size $\log^A x$ for any fixed $A>0$, but we will not work with such small sizes of $y$ here.

It is well-known that these questions are tied to zero density theorems for the Riemann zeta function $\zeta(s)$.
For $0 \le \sigma < 1$ and $T>0$, let $\mathcal{Z}(\sigma,T)$ denote the (multi-)set
\begin{equation}
\mathcal{Z}(\sigma, T) \coloneqq  \left\{ \rho = \beta + i\gamma : \zeta(\rho)=0, \beta\ge \sigma, |\gamma|\le T \right\}
\end{equation}
of zeroes of $\zeta$ (counted with multiplicity), and let $N(\sigma, T) = \#\mathcal{Z}(\sigma, T)$ be the number of such zeroes (again counting multiplicity).  For any $0 \leq \sigma < 1$, let $A(\sigma)$ denote the least exponent for which one has the zero density theorem
$$N(\sigma, T) \ll_{\sigma, \eps} T^{A(\sigma)(1-\sigma) + \eps}$$
for any $\eps>0$, or equivalently that
\begin{equation}\label{nsig-up}
  N(\sigma, T) \leq T^{A(\sigma)(1-\sigma) + o(1)}
\end{equation}
as $T \to \infty$ (holding $\sigma$ fixed).  If $N(\sigma,T)$ vanishes for all $T$, then we adopt the convention that $A(\sigma)=-\infty$.

For $0 \leq \sigma \leq 1/2$, the Riemann--von Mangoldt formula together with the functional equation give $N(\sigma,T) \asymp T \log T$ for all large $T$, hence
\begin{equation}\label{asig-small}
A(\sigma) = \frac{1}{1-\sigma}
\end{equation}
in this regime.  For $1/2 < \sigma \leq 1$ the situation is not fully resolved; see \Cref{zero_density_table} and \Cref{fig:a_sigma} for the current best upper bounds on $A(\sigma)$.  The \emph{density hypothesis} asserts that $A(\sigma) \leq 2$ for all $\sigma$; this is currently only known for $0 \leq \sigma \leq 1/2$ and $25/32 < \sigma \leq 1$ \cite{bourgain-density}, although it is also known to follow from the Lindel\"of hypothesis (LH) \cite{ingham}.  The Lindel\"of hypothesis also implies that $A(\sigma) \leq 0$ for $3/4 < \sigma \leq 1$ \cite{halasz}.  Finally, the Riemann hypothesis of course implies that $A(\sigma)=-\infty$ for all $1/2 < \sigma < 1$.  We also make the trivial observation that $A(\sigma) (1-\sigma)$ is non-increasing in $\sigma$.

\begin{figure}
  \centering
  \includegraphics[width=0.8\textwidth]{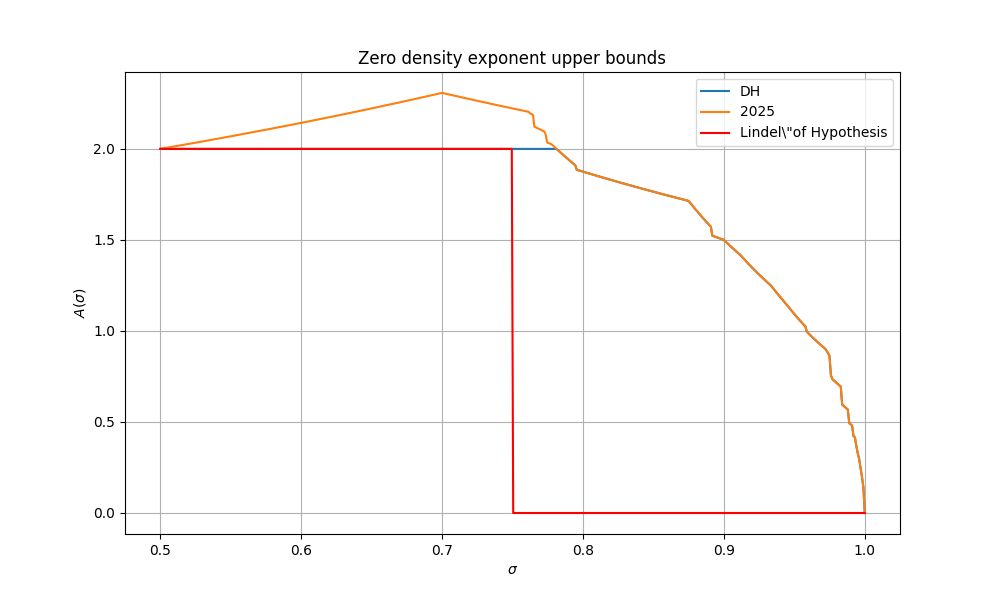}
  \caption{Current best upper bounds on $A(\sigma)$ unconditionally; assuming DH; and assuming LH. 
  }\label{fig:a_sigma}
  \end{figure}

\begin{table}[ht]
\def\arraystretch{1.2}
    \centering
    \small
    \begin{tabular}{|c|c|c|}
    \hline
    $A(\sigma)$ bound & Range & Reference\\
    \hline
    $\dfrac{3}{2 - \sigma}$ & $\dfrac{1}{2} \leq \sigma \le \dfrac{7}{10} = 0.7$ & Ingham \cite{ingham} \\
    \hline
    $\dfrac{15}{3+5\sigma}$ & $\dfrac{7}{10} \leq \sigma < \dfrac{19}{25} = 0.76$ & Guth--Maynard \cite{GuthMaynard}\\
    \hline
    $\dfrac{9}{8\sigma - 2}$ & $\dfrac{19}{25} \leq \sigma < \dfrac{127}{167} = 0.7604\ldots$ & Ivic \cite{ivic} \\
    \hline
    $\dfrac{15}{13\sigma - 3}$ & $\dfrac{127}{167} \leq \sigma < \dfrac{13}{17} = 0.7647\ldots$ & Ivic \cite{ivic}\\
    \hline
    $\dfrac{6}{5\sigma - 1}$ & $\dfrac{13}{17} \leq \sigma < \dfrac{17}{22} = 0.7727\ldots$ & Ivic \cite{ivic}\\
    \hline
    $\dfrac{2}{9\sigma - 6}$ & $\dfrac{17}{22} \leq \sigma < \dfrac{41}{53} = 0.7735\ldots$ & Tao--Trudgian--Yang \cite{tty}\\
    \hline
    $\dfrac{9}{7\sigma - 1}$ & $\dfrac{41}{53} \leq \sigma < \dfrac{7}{9} = 0.7777\ldots$ & Ivic \cite{ivic}\\
    \hline
    $\dfrac{9}{8(2\sigma-1)}$ & $\dfrac{7}{9} \le \sigma < \dfrac{1867}{2347} = 0.7954\ldots$ & Tao--Trudgian--Yang \cite{tty} \\
    \hline
    $\dfrac{3}{2\sigma}$ & $\dfrac{1867}{2347} \le \sigma < \dfrac{4}{5} = 0.8$ & Bourgain \cite{bourgain} \\
    \hline
    $\dfrac{3}{2\sigma}$ & $\dfrac{4}{5} \le \sigma < \dfrac{7}{8} = 0.875$ & Ivic \cite{ivic} \\
    \hline
    $\dfrac{3}{10\sigma - 7}$ & $\dfrac{7}{8} \le \sigma < \dfrac{279}{314} = 0.8885\ldots$ & Heath--Brown \cite{HB-2} \\
    \hline
    $\dfrac{24}{30\sigma - 11}$ & $\dfrac{279}{314} \le \sigma < \dfrac{155}{174} = 0.8908\ldots$ & CDV \cite{cdv} \\
    \hline
    $\dfrac{24}{30\sigma - 11}$& $\dfrac{155}{174} \le \sigma \le \dfrac{9}{10} = 0.9$ & Ivic \cite{ivic}\\
    \hline
    $\dfrac{3}{10\sigma - 7}$ & $\dfrac{9}{10} < \sigma \le \dfrac{31}{34} = 0.9117\ldots$ & Tao--Trudgian--Yang \cite{tty}\\
    \hline
    $\dfrac{11}{48\sigma - 36}$ & $\dfrac{31}{34} < \sigma < \dfrac{14}{15} = 0.9333\ldots$ & Tao--Trudgian--Yang \cite{tty}\\
    \hline
    $\dfrac{391}{2493\sigma - 2014}$ & $\dfrac{14}{15} \le \sigma < \dfrac{2841}{3016} = 0.9419\ldots$ & Tao--Trudgian--Yang \cite{tty}\\
    \hline
    $\dfrac{22232}{163248\sigma - 134765}$ & $\dfrac{2841}{3016} \le \sigma < \dfrac{859}{908} = 0.9460\ldots$ & Tao--Trudgian--Yang \cite{tty}\\
    \hline
    $\dfrac{356}{2742\sigma - 2279}$ & $\dfrac{859}{908} \le \sigma < \dfrac{23}{24} = 0.9583\ldots$ & Tao--Trudgian--Yang \cite{tty}\\
    \hline
    $\dfrac{3}{24\sigma - 20}$ & $\dfrac{23}{24} \leq \sigma < \dfrac{2211487}{2274732} = 0.9721\ldots$ & Pintz \cite{pintz} \\
    \hline
    $\dfrac{86152}{1447460\sigma - 1311509}$ & $\dfrac{2211487}{2274732} \le \sigma < \dfrac{39}{40} = 0.975$ & Tao--Trudgian--Yang \cite{tty} \\
    \hline
    $\dfrac{2}{15\sigma - 12}$ & $\dfrac{39}{40} \leq \sigma < \dfrac{41}{42} = 0.9761\ldots$ & Pintz \cite{pintz} \\
    \hline
    $\dfrac{3}{40 \sigma - 35}$ & $\dfrac{41}{42} \leq \sigma < \dfrac{59}{60} = 0.9833\ldots$ & Pintz \cite{pintz} \\
    \hline
    $\dfrac{3}{n(1 - 2(n - 1)(1 - \sigma))}$ & \begin{tabular}{@{}c@{}}$1 - \dfrac{1}{2n(n - 1)} \le \sigma < 1 - \dfrac{1}{2n(n + 1)}$\\(for integer $n \ge 6$)\end{tabular} & Pintz \cite{pintz} \\
    \hline
    \end{tabular}
    \caption{Current best upper bounds on $A(\sigma)$, taken from the ANTEDB \cite{antedb}.}  \label{zero_density_table}

    \end{table}

The following result is standard (see, e.g., \cite[\S 13]{GuthMaynard} or \cite[Theorem 10.5, Exercise 10.5.6]{ik}), and will also be reproven here:

\begin{theorem}[Zero density theorems and PNT in short intervals]\label{folklore}  Let $A_0>0$ be such that
  \begin{equation}\label{asigma}
    A(\sigma) \leq A_0
  \end{equation}
   for all $1/2 \leq \sigma < 1$.  Let $0 < \theta < 1$ be fixed, and set $y = x^\theta$.
  \begin{itemize}
    \item[(i)] (PNT in all short intervals) If $\theta > 1-\frac{1}{A_0}$, then \eqref{xyn} holds for all $x$.
    \item[(ii)] (PNT in almost all short intervals) If $\theta > 1-\frac{2}{A_0}$, then \eqref{xyn} holds for all $x$ outside of an exceptional set of density zero.
  \end{itemize}
\end{theorem}

For instance, from \Cref{zero_density_table} one can obtain the upper bound $A(\sigma) \leq \frac{30}{13}$ for all $1/2 \leq \sigma < 1$, and hence one has the PNT in all intervals for $\theta > \frac{17}{30}$ and for almost all intervals for $\theta > \frac{2}{15}$; see \cite[Corollary 1.3, Corollary 1.4]{GuthMaynard}.  Assuming the density hypothesis, these ranges can be improved to $\theta > \frac{1}{2}$ and $\theta > 0$ respectively.

\subsection{Bounding the exceptional set}

In this paper we are interested in bounding the size of the exceptional set where \eqref{xyn} fails.  This problem was studied in \cite{baz-1}, \cite{bp}, \cite{baz-2}, and we now reproduce some notation from these papers.

\begin{definition}[Exceptional set exponent]\label{excep}\cite{bp}, \cite{baz-2}  Fix $0 < \theta < 1$.  For any $X > 1$ and $\delta > 0$, let ${\mathcal E}_\delta(X,\theta)$ denote the set of all $x \in [X,2X]$ such that
  $$ \left|\sum_{x < n \leq x+y} \Lambda(n) - y \right| \geq \delta y$$
where $y \coloneqq x^\theta$, and then let $\mu_\delta(\theta)$ denote the infimum of all exponents $\xi$ such that
$$ |{\mathcal E}_\delta(X,\theta)| \ll_{\delta,\theta} X^{\xi}$$
for all sufficiently large $X$, where $|{\mathcal E}_\delta(X,\theta)|$ denotes the Lebesgue measure of ${\mathcal E}_\delta(X,\theta)$; we adopt the convention that $\mu_\delta(\theta)=-\infty$ if ${\mathcal E}_\delta(X,\theta)$ vanishes for all sufficiently large $X$.  Finally, we set\footnote{This function $\mu(\theta)$ should not be confused with the M\"obius function, or the growth exponent for the Riemann zeta function $\zeta(\sigma+it)$ for a fixed $\sigma$.}
$$ \mu(\theta) \coloneqq \sup_{\delta > 0} \mu_\delta(\theta).$$
\end{definition}

Unpacking the definitions (setting $\delta = 1, 1/2, 1/3,\dots$ and performing a suitable diagonalization), we see that $\mu(\theta)$ is the least exponent for which the prime number theorem in short intervals \eqref{xyn} with $y = x^\theta$ holds for all $x \geq 1$ outside of an exceptional set ${\mathcal E}_\theta \subset [1,+\infty)$ with
$$ |{\mathcal E}_\theta \cap [X, 2X]| \ll X^{\mu(\theta)+o(1)}$$
as $X \to \infty$ (holding $\theta$ fixed). In particular, if $\mu(\theta) < 1$, then \eqref{xyn} would hold outside of a density zero exceptional set of parameters $x>1$.  It is easy to show that $\mu$ is a non-increasing function of $\theta$; see \cite[Corollary 1(i)]{bp}.  The exponent $\mu(\theta)$ also gives information about prime gaps: if $p_n$ denotes the $n^{\mathrm{th}}$ prime, it is easy to see that $p_{n+1}-p_n \leq 2 p_n^{\theta}$ (say) for all $n$ outside of an exceptional set ${\mathcal N}_\theta$ with\footnote{In view of results such as \cite{islam}, \cite{peck}, \cite{jarviniemi}, \cite{li}, it should be possible to improve upon \eqref{nn} by combining the methods in this paper with sieves such as the Harman sieve, but we will not attempt to do so here.}
\begin{equation}\label{nn}
   |{\mathcal N}_\theta \cap [N, 2N]| \ll N^{\mu(\theta)-\theta+o(1)}
\end{equation}
as $N \to \infty$, since any large prime gap $p_{n+1}-p_n > 2 p_n^{\theta}$ will lead to a violation of the prime number theorem \eqref{xyn} in intervals $[x,x+x^\theta]$ for $p_n \leq x \leq p_n + p_n^\theta$.

From \Cref{folklore}(i), we have $\mu(\theta) = -\infty$ for $\theta > 1 - \frac{1}{A_0}$.  It is then reasonable to expect upper bounds on $\mu(\theta)$ that are between $0$ and $1$ in the range $1-\frac{2}{A_0} < \theta \leq 1-\frac{1}{A_0}$.  In this regard, we have the following results in the literature\footnote{We thank Kaisa Matom\"aki for these and other references.}.

\begin{lemma}[Previous bounds on $\mu$]\label{baz-bound}\
  \begin{itemize}
    \item[(i)]\cite[Theorem 2(i)]{bp} For sufficiently small $\Delta>0$, we have $\mu(1/6 + \Delta) \leq 1 - c\Delta$ and $\mu(7/12-\Delta) \leq \frac{5}{8} + \frac{7}{4}\Delta + O(\Delta^2)$ for some absolute constant $c>0$.
  \item[(ii)] \cite[Theorem 2(ii)]{bp} Assuming RH, we have $\mu(\theta) \leq 1-\theta$ for $0 < \theta \leq 1/2$.
    \item[(iii)] \cite[Lemma 1]{baz-2} We have
  $$
  \mu(\theta) \leq
  \begin{cases}
  \frac{3(1-\theta)}{2} & \frac{1}{2} < \theta \leq \frac{11}{21} \\
  \frac{47-42\theta}{35}& \frac{11}{21} < \theta \leq \frac{23}{42} \\
  \frac{36\theta^2-96\theta+55}{39-36\theta} & \frac{23}{42} < \theta \leq \frac{7}{12} \\
  \end{cases}
  $$
  \end{itemize}
\end{lemma}

In \cite[Lemma 1]{baz-2}, further bounds were claimed in the region $\frac{1}{6} < \theta \leq \frac{1}{2}$, but unfortunately the argument is incomplete in that range; see \Cref{numerics-sec} for more discussion.  We also note that there is extensive literature \cite{selberg}, \cite{wolke}, \cite{cook}, \cite{huxley}, \cite{yu}, \cite{HB-1}, \cite{HB-III}, \cite{HB-2}, \cite{ivic-2}, \cite{peck}, \cite{matomaki}, \cite{HB-3}, \cite{stadlmann}, \cite{jarviniemi} on related quantities such as $\sum_{p_n \leq x} (p_{n+1}-p_n)^2$.  For instance, in \cite{selberg} it was shown that RH implies the bound $\sum_{n \leq x} (p_{n+1}-p_n)^2 \ll x \log^3 x$, which (up to $x^{o(1)}$ errors) can be recovered from \Cref{baz-bound}(ii) using \eqref{nn}.  We will not attempt to summarize the rest of this literature here, but refer the reader to \cite{antedb} instead.

The results in \Cref{baz-bound}(i) relied on the best available zero density estimates at the time, which among other things could only establish \eqref{asigma} for $A_0=12/5$.  In fact, the methods in that paper give a general relationship between $\mu$ and $A$, which is our first main result:

  \begin{theorem}[General bound]\label{gen-bound}  For any $0 < \theta < 1$, one has
\begin{equation}\label{muth}
   \mu(\theta) \leq \inf_{\eps>0} \sup_{\stackrel{0 \leq \sigma < 1}{A(\sigma) \geq \frac{1}{1-\theta}-\eps}} \mu_{2,\sigma}(\theta)
\end{equation}
where
$$ \mu_{2,\sigma}(\theta) \coloneqq (1-\theta)(1-\sigma) A(\sigma) + 2\sigma - 1.$$
Here we adopt the convention that the empty supremum is $-\infty$.
\end{theorem}

We remark that if $\tilde A(\sigma)$ is any upper bound for $A(\sigma)$ that depends continuously\footnote{One can also take $\tilde A$ to be the left-continuous majorant $\tilde A(\sigma) \coloneqq \limsup_{\sigma' \to \sigma} A(\sigma')$, which is the version of $A(\sigma)$ that is used in \cite{tty}.} on $\sigma$, then \eqref{muth} implies the simpler upper bound
\begin{equation}\label{muth-simp}
   \mu(\theta) \leq \sup_{\stackrel{0 \leq \sigma < 1}{\tilde A(\sigma) \geq \frac{1}{1-\theta}}} \tilde \mu_{2,\sigma}(\theta)
\end{equation}
where $\tilde \mu_{2,\sigma}$ is defined as $\mu_{2,\sigma}$ but with $A(\sigma)$ replaced by $\tilde A(\sigma)$.  In practice, the known upper bounds on $A(\sigma)$ are continuous in $\sigma$, so morally speaking one can set $\eps$ to zero in the above formula.  However, it is not currently known if $A(\sigma)$ itself is continuous, so we could not rigorously drop the $\eps$ parameter from the above result.

This theorem already recovers several previously mentioned results:
\begin{itemize}
  \item For $\theta > 1-\frac{1}{A_0}$, the supremum in \eqref{muth} is empty for $\eps$ small enough, and we recover \Cref{folklore}(i).
  \item For $\theta > 1-\frac{2}{A_0}$, the quantity $\mu_{2,\sigma}(\theta)$ is bounded above by $1-c(1-\sigma)$ for some $c>0$, and $A(\sigma)$ is also known (see, e.g., \cite{turan}) to go to zero as $\sigma \to 1$, and so we can also easily recover \Cref{folklore}(ii).
  \item Assuming the Riemann hypothesis, we can restrict attention to the range $0 \leq \sigma \leq 1/2$, and then \eqref{asig-small} readily gives \Cref{baz-bound}(ii).  Unconditionally, the same argument gives the bound
 \begin{equation}\label{muth-12}
   \mu(\theta) \leq \max \left( 1-\theta, \inf_{\eps>0} \sup_{\stackrel{1/2 < \sigma < 1}{A(\sigma) \geq \frac{1}{1-\theta}-\eps}} \mu_{2,\sigma}(\theta) \right).
 \end{equation}
\end{itemize}

However, \Cref{gen-bound} is not quite strong enough to recover the previous bounds in \Cref{baz-bound}(i).  This is because the arguments in that paper also relied on additional bounds on the \emph{additive energy} of zeroes, which were implicitly\footnote{The terminology ``additive energy'' was only introduced several decades later in \cite{tao-vu}.} introduced in \cite{HB-1}, and which we now define.  For $\frac12 \le \sigma < 1$ and $T>0$, let $N^*(\sigma,T)$ be given by
$$N^*(\sigma, T) =  \#\left\{ (\rho_1, \rho_2, \rho_3, \rho_4) : \rho_j = \beta_j + i\gamma_j\in \mathcal{Z}(\sigma, T), |\gamma_1 + \gamma_2 - \gamma_3 - \gamma_4|\le1 \right\}.$$
 For any $\frac{1}{2} \leq \sigma < 1$, let $A^*(\sigma)$ denote the least exponent for which one has the zero density additive energy theorem
$$N^*(\sigma, T) \ll_{\sigma, \eps} T^{A*(\sigma)(1-\sigma) + \eps}$$
for any $\eps>0$, or equivalently that
$$N^*(\sigma, T) \leq T^{A^*(\sigma)(1-\sigma) + o(1)}$$
as $T \to \infty$ (holding $\sigma$ fixed).  For $0 \leq \sigma \leq 1/2$, it is not difficult to use the Riemann--von Mangoldt formula and the functional equation to obtain the bounds
$$ N^*(\sigma,T) \asymp T^3\log^4 T$$
and hence
$$ A^*(\sigma) = \frac{3}{1-\sigma}.$$
For $1/2 < \sigma\leq 1$, the known bounds on $A^*(\sigma)$ are summarized in \Cref{zero-density-energy-table} and \Cref{fig:astar_sigma}.  One always has the trivial bound $A^*(\sigma) \leq 3 A(\sigma)$, but in many ranges of $\sigma$, stronger bounds are known.

\begin{figure}
  \centering
  \includegraphics[width=0.8\textwidth]{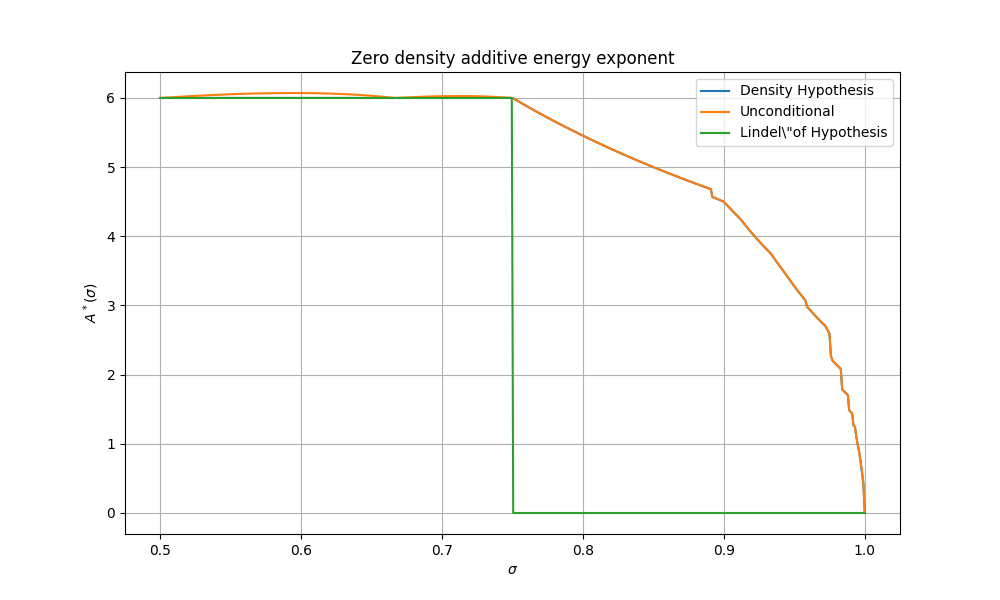}
  \caption{Current best bounds on $A^*(\sigma)$ unconditionally; assuming DH; and assuming LH.  The DH bound is the minimum of the unconditional bound and $6$, and is obscured by the other two graphs in this plot.}\label{fig:astar_sigma}
  \end{figure}

\begin{table}[ht]
    \def\arraystretch{1.2}
    \centering
\small
    \begin{tabular}{|c|c|c|}
    \hline
    $A^*(\sigma)$ bound & Range & Reference\\
    \hline
    $\dfrac{10 - 11\sigma}{(2 - \sigma)(1 - \sigma)}$ & $\dfrac{1}{2} \leq \sigma \le \dfrac{2}{3} = 0.6666\ldots$ & Heath-Brown \cite{HB-2}\\
    \hline
    $\dfrac{18 - 19\sigma}{(4 - 2\sigma)(1 - \sigma)}$ & $\dfrac{2}{3} \leq \sigma \le \dfrac{7}{10} = 0.7$ & Heath-Brown \cite{HB-2}\\
    \hline
    $\dfrac{5(18 - 19\sigma)}{2(5\sigma + 3)(1 - \sigma)}$ & $\dfrac{7}{10} \leq \sigma \le \dfrac{539 - \sqrt{42121}}{460} = 0.7255\ldots$ & Tao--Trudgian--Yang \cite{tty}\\
    \hline
    $\dfrac{2(45 - 44\sigma)}{(2\sigma + 15)(1 - \sigma)}$ & $\dfrac{539 - \sqrt{42121}}{460} \leq \sigma \le \dfrac{165}{226} = 0.7300\ldots$ & Tao--Trudgian--Yang \cite{tty}\\
    \hline
    $\dfrac{457 - 546\sigma}{2(61 - 58\sigma)(1-\sigma)}$ & $\dfrac{165}{226} \leq \sigma \le \dfrac{5831 + \sqrt{60001}}{8240} = 0.7373\ldots$ & Tao--Trudgian--Yang \cite{tty}\\
    \hline
    $\dfrac{5(18 - 19\sigma)}{2(5\sigma + 3)(1 - \sigma)}$ & $\dfrac{5831 + \sqrt{60001}}{8240} \leq \sigma \le \dfrac{42}{55} = 0.7636\ldots$ & Tao--Trudgian--Yang \cite{tty}\\
    \hline
    $\dfrac{18 - 19\sigma}{6(15\sigma - 11)(1- \sigma)}$ & $\dfrac{42}{55} \leq \sigma \le \dfrac{97}{127} = 0.7637\ldots$ & Tao--Trudgian--Yang \cite{tty}\\
    \hline
    $\dfrac{3(18-19\sigma)}{4(4\sigma-1)(1 - \sigma)}$ & $\dfrac{97}{127} \leq \sigma \le \dfrac{79}{103} = 0.7669\ldots$ & Tao--Trudgian--Yang \cite{tty}\\
    \hline
    $\dfrac{18 - 19\sigma}{2(37\sigma - 27)(1 - \sigma)}$ & $\dfrac{79}{103} \leq \sigma \le \dfrac{33}{43} = 0.7674\ldots$ & Tao--Trudgian--Yang \cite{tty}\\
    \hline
    $\dfrac{5(18 - 19\sigma)}{2(13\sigma - 3)(1 - \sigma)}$ & $\dfrac{33}{43} \leq \sigma \le \dfrac{84}{109} = 0.7706\ldots$ & Tao--Trudgian--Yang \cite{tty}\\
    \hline
    $\dfrac{18 - 19\sigma}{9(3\sigma - 2)(1 - \sigma)}$ & $\dfrac{84}{109} \leq \sigma \le \dfrac{1273 - \sqrt{128689}}{1184} = 0.7721\ldots$ & Tao--Trudgian--Yang \cite{tty}\\
    \hline
    $\dfrac{4(10 - 9\sigma)}{5(4\sigma - 1)(1 - \sigma)}$ & $\dfrac{1273 - \sqrt{128689}}{1184} \leq \sigma \le \dfrac{5}{6} = 0.8333\ldots$ & Tao--Trudgian--Yang \cite{tty}\\
    \hline
    $\dfrac{12}{4\sigma - 1}$ & $\dfrac{5}{6} \leq \sigma \le 1$ & Heath-Brown \cite{HB-2}\\
    \hline
    \end{tabular}
 
    \caption{Current best upper bounds on $A^*(\sigma)$, taken from the ANTEDB \cite{antedb}.}   \label{zero-density-energy-table}
\end{table}

We then have the following refinement of \Cref{gen-bound}:

\begin{theorem}[Refined bound]\label{refined-bound}  For any $0 < \theta < 1$, one has
$$ \mu(\theta) \leq \inf_{\eps>0} \sup_{\stackrel{0 \leq \sigma < 1}{A(\sigma) \geq \frac{1}{1-\theta}-\eps}} \min( \mu_{2,\sigma}(\theta), \mu_{4,\sigma}(\theta))$$
where
$$ \mu_{4,\sigma}(\theta) \coloneqq (1-\theta)(1-\sigma) A^*(\sigma) + 4\sigma-3.$$
\end{theorem}

As in \eqref{muth-12}, we can also bound
$$ \mu(\theta) \leq \max\left(1-\theta, \inf_{\eps>0} \sup_{\stackrel{1/2 < \sigma < 1}{A(\sigma) \geq \frac{1}{1-\theta}-\eps}} \min( \mu_{2,\sigma}(\theta), \mu_{4,\sigma}(\theta)) \right).$$

We establish \Cref{refined-bound} in \Cref{proof-sec}, following the strategy of previous works (particularly \cite{HB-2}, \cite{baz-2}) of applying the explicit formula for $\sum_{n \leq x} \Lambda(n)$, followed by $L^\infty$, $L^2$, and $L^4$ estimates for various components of that formula.

In \Cref{numerics-sec} we report on the bounds obtained by combining \Cref{refined-bound} with the estimates in \Cref{zero_density_table} and \Cref{zero-density-energy-table}.

\subsection{Acknowledgments}

AG is supported by NSF grant OIA-2229278. TT is supported by NSF grant DMS-2347850. The authors thank Kaisa Matom\"aki for providing several relevant references.

\subsection{Notation}\label{notation-sec}

We use the following asymptotic notation throughout the paper:

\begin{itemize}
\item We use $X \ll Y$, $Y \gg X$, or $X = O(Y)$ to denote the estimate $|X| \leq C Y$ for some constant $C>0$.  If we need this constant $C$ to depend on parameters, we will indicate this by subscripts, thus for instance $X \ll_\theta Y$ denotes the bound $|X| \leq C_\theta Y$ for some $C_\theta$ depending on $Y$.  We use $X \asymp Y$ to denote the estimate $X \ll Y \ll X$.
\item We use $Y = o(Z)$ as $X \to \infty$ to denote the estimate $|Y| \leq c(X) Z$ for some function $c(X)$ of a parameter $X$ that goes to zero as $X \to \infty$ (holding all other parameters fixed).  We use $Y \sim Z$ to denote the estimate $Y = (1+o(1)) Z$ as $X \to \infty$.
\end{itemize}

\section{Proof of bound}\label{proof-sec}

In this section we prove \Cref{refined-bound}, which of course implies \Cref{gen-bound} as a corollary.

We begin with some simple reductions.
Fix $0 <\theta < 1$ and $\eps$, and let $\mu$ denote the quantity
$$ \mu \coloneqq \sup_{\stackrel{0 \leq \sigma < 1}{A(\sigma) \geq \frac{1}{1-\theta}-\eps}} \min( \mu_{2,\sigma}(\theta), \mu_{4,\sigma}(\theta)),$$
thus
whenever $0 \leq \sigma < 1$ is such that $A(\sigma) \geq \frac{1}{1-\theta}-\eps$, at least one of the inequalities
\begin{equation}\label{muth-alt}
(1-\theta)(1-\sigma) A(\sigma) + 2\sigma - 1 \leq \mu
\end{equation}
and
\begin{equation}\label{muth-alt-2}
(1-\theta)(1-\sigma) A^*(\sigma) + 4\sigma - 3 \leq \mu
\end{equation}
hold. Our task is then to show that $\mu(\theta) \leq \mu$.

Fix $0 < \delta < 1$.  It will suffice to show that whenever $J$ is a natural number that is sufficiently large (depending on $\delta, \theta, \eps$), that one has the upper bound
$$|{\mathcal E}_\delta(X,\theta)| \ll_{\delta,\theta,J} X^{\mu+4/J+o(1)}$$
as $X \to \infty$ (holding the other parameters $J,\delta,\theta,\eps$ fixed).  We may of course assume that $X$ is larger than any specified absolute constant.

By subdividing $[X,2X]$ into $O_\delta(1)$ intervals of the form $[X,(1+\delta/J) X]$ and using the triangle inequality, it suffices to show that
$$\left|\left\{ X \leq x < (1+\delta/J) X: \left| \sum_{x < n \leq x + x^\theta} \Lambda(n) - x^\theta \right| \geq \delta x^\theta \right\}\right| \ll_{\delta,\theta,J} X^{\mu+4/J+o(1)}.$$
From the Brun--Titchmarsh inequality we see that for $X \leq x < (1+\delta/J) X$, one has
$$ \sum_{x < n \leq x + x^\theta} \Lambda(n) = \sum_{x < n \leq x + x/\tau} \Lambda(n) + O\left(\frac{\delta}{J} x^\theta\right)$$
where $\tau \coloneqq X^{1-\theta}$, so in particular $x/\tau = x^\theta +  O(\frac{\delta}{J} x^\theta)$.  Thus, by the triangle inequality, it will suffice (after choosing $J$ large enough) to show that
\[
\left| \left\{ X \leq x < (1+\delta/J) X : \left| \sum_{x < n \leq x + x/\tau} \Lambda(n) - \frac{x}{\tau} \right| \geq \frac{\delta}{2} \frac{X}{\tau} \right\} \right| \ll_{\delta,\theta,J,\eps} X^{\mu+4/J+o(1)}.
\]

We set $T \coloneqq J (\log^2 X) \tau = X^{1-\theta+o(1)}$.  By the explicit formula (see, e.g., \cite[Chapter 17]{davenport}) we have
$$ \sum_{x < n \leq x + x/\tau} \Lambda(n) - \frac{x}{\tau} = \sum_{\rho \in {\mathcal Z}(0,T)} \frac{(x+x/\tau)^\rho-x^\rho}{\rho} + O\left( \frac{x (\log x)^2}{T} \right).$$
By choice of $T$, the error term here is $O(\frac{1}{J} X^\theta)$.  Thus, for $J$ large enough, it suffices to show that
\begin{equation}\label{targ}
  \left|\left\{ X \leq x < (1+\delta/J) X: |S_{[0,1]}(x)| \geq \frac{\delta}{3} \frac{X}{\tau} \right\}\right| \ll_{\delta,\theta,J,\eps} X^{\mu+4/J+o(1)}
\end{equation}
where, for any subset $I$ of $[0,1]$ and any $x > 1$, we write
\begin{equation}\label{six}
   S_I(x) \coloneqq \sum_{\rho \in {\mathcal Z}(0,T): \Re \rho \in I} \frac{(x+x/\tau)^\rho-x^\rho}{\rho}.
\end{equation}

We first dispose of the region near the right edge $\sigma=1$ of the critical line.

\begin{lemma}[Contribution near the right edge]\label{right-edge}  If $\eta_0>0$ is sufficiently small depending on $\theta$, and $I \subset [1-\eta_0,1]$, then
  \begin{equation}\label{si-bound}
\sup_{X \leq x \leq 2X}|S_I(x)| \ll_\theta \exp\left(-c_\theta \frac{\log^{1/3} X}{\log\log^{1/3} X} \right)\frac{X}{\tau}
  \end{equation}
  for some $c_\theta>0$ depending only on $\theta$.
\end{lemma}

\begin{proof} This bound is essentially in \cite[\S 3]{HB-2} (with the explicit choice $\eta_0 = 10^{-7}$), but we give a proof here for the convenience of the reader.
  From the fundamental theorem of calculus and the triangle inequality, we have
\begin{align*}
\sup_{X \leq x \leq 2X}|S_I(x)| &\ll \sum_{\sigma+it \in {\mathcal Z}(1-\eta_0,T)} \frac{X}{\tau} X^{\sigma-1} \\
&\ll \frac{X}{\tau} \sum_{\sigma+it \in {\mathcal Z}(1-\eta_0,T)} X^{-\eta_0} + \int_{1-\sigma}^{\eta_0} (\log X) X^{-\eta}\ d\eta \\
&\ll \frac{X \log X}{\tau} \left( N(1-\eta_0,T) X^{-\eta_0} + \int_{0}^{\eta_0} N(1-\eta,T) X^{-\eta}\ d\eta \right)\\
&\ll \frac{X \log X}{\tau} \sup_{0 < \eta \leq \eta_0} N(1-\eta,T) X^{-\eta}.
\end{align*}
From the Vinogradov--Korobov zero-free region, we see that $N(1-\eta,T)$ vanishes unless
\begin{equation}\label{eta-vanish}
  \eta \gg_\theta \frac{1}{\log^{2/3} X \log\log^{1/3} X}.
\end{equation}
We also have the upper bound
$$ N(1-\eta,T) \ll T^{C \eta^{3/2}} \log^{O(1)} T \ll X^{C \sigma \eta^{3/2}} \log^{O(1)} X$$
for some absolute constant $C>0$; such a bound was first obtained (with $C = 12000$) in \cite[Theorem 38.2]{turan}.  The value of $C$ has improved substantially, for instance\footnote{In principle, one should be able to lower $C$ to $3\sqrt{2}$ to match the results in \cite{pintz}, but the zero density theorems there contain $T^\eps$ type losses which are unacceptable for this portion of the argument. In fact, the bounds in \cite{pintz} can be used to show that one can take $\eta_0$ to be any quantity between $0$ and $1/24$; we leave this calculation to the interested reader.} to $C=58.05$ in \cite{ford}, but the precise value of $C$ is not of importance for our purposes.  For $\eta_0$ small enough depending on $\sigma$, this then gives
$$ N(1-\eta,T) X^{-\eta} \ll X^{-\eta/2} \log^{O(1)} X$$
which, when combined with the vanishing of this expression outside of the region \eqref{eta-vanish}, gives \eqref{si-bound}.
\end{proof}

A similar argument controls intervals $I$ away from the right edge:

\begin{lemma}[$L^\infty$ bound]\label{linfty-bound}  If $I \subset [\sigma_-, \sigma_+]$ for some fixed $0 \leq \sigma_- \leq \sigma_+ \leq 1$, then
$$ \sup_{X \leq x \leq 2X} |S_I(x)| \ll_\theta X^{(1-\theta)(1-\sigma_-) A(\sigma_-) + \theta + \sigma_+ - 1 + o(1)}$$
as $X \to \infty$.
\end{lemma}

\begin{proof}
By repeating the arguments used to prove \Cref{right-edge}, we have
$$ S_I(x) \ll \frac{X \log X}{\tau} \sup_{\sigma_- \leq \sigma \leq \sigma_+} N(\sigma,T) X^{\sigma-1}.$$
If we crudely bound $N(\sigma,T) \leq N(\sigma_-,T)$ and $X^{\sigma-1} \leq X^{\sigma_+ - 1}$, we obtain
$$ S_I(x) \ll \frac{X \log X}{\tau} N(\sigma_-,T) X^{\sigma_+ - 1}.$$
The claim now follows from \eqref{nsig-up} and the definitions of $T,\tau$.
\end{proof}

We also have a second moment bound:

\begin{lemma}[$L^2$ bound]\label{l2-bound}  If $I \subset [\sigma_-, \sigma_+]$ for some fixed $0 \leq \sigma_- \leq \sigma_+ \leq 1$, then
$$ \frac{1}{X}\int_X^{2X} |S_I(x)|^2\ dx \ll_\theta X^{(1-\theta) (1-\sigma_-) A(\sigma_-) + 2\theta + 2\sigma_+ - 2 + o(1)}$$
\end{lemma}

\begin{proof}  This estimate is essentially in \cite[\S 2]{HB-2}, but for the convenience of the reader we give a proof here.  We can bound the left-hand side by
  $$ \ll \int_0^\infty |S_I(x)|^2 \psi(\log x - \log X)\ \frac{dx}{x}$$
  where $\psi$ is fixed non-negative bump function that is bounded away from zero on $[0,\log 2]$.  Using \eqref{six}, we may expand this as
$$ \sum_{\sigma+it, \sigma'+it' \in {\mathcal Z}(0,T): \sigma,\sigma' \in I} c_{\sigma+it} \overline{c}_{\sigma'+it'} X^{\sigma+\sigma'+i(t-t')} \hat \psi( t-t' - i(\sigma+\sigma'))$$
where
$$ \hat \psi(\xi) \coloneqq \int_\R \psi(u) e^{iu\xi}\ du$$
is the (complexified) Fourier transform of $\psi$, and $c_{\sigma+it}$ are the coefficients
$$ c_{\sigma+it} \coloneqq \frac{(1+1/\tau)^{\sigma+it} - 1}{\sigma+it}.$$
By standard integration by parts estimates, we have
$$ \hat \psi( t-t' - i(\sigma+\sigma')) \ll\frac{1}{(1 + |t-t'|)^{10}}$$
(say), while from the fundamental theorem of calculus one has
$$ c_{\sigma+it}, c_{\sigma'+it'} \ll \frac{1}{\tau} = X^{\theta - 1 + o(1)}.$$
By the triangle inequality, we conclude that
$$ \frac{1}{X}\int_X^{2X} |S_I(x)|^2\ dx
\ll X^{2\theta+ 2\sigma_+-2+o(1)} \sum_{\sigma+it, \sigma'+it' \in {\mathcal Z}(\sigma_-,T)} \frac{1}{(1 + |t-t'|)^{10}}.$$
From the Riemann-von Mangoldt formula, there are there are at most $O(\log T)$ zeroes $\sigma+it \in {\mathcal Z}(\sigma_-,T)$ (counting multiplicity) with $t$ in any given unit interval in $[-T,T]$, so we can bound the above expression by
$$ \ll X^{2\theta+ 2\sigma_+-2+o(1)} (\log T) N(\sigma_-, T),$$
and the claim now follows from \eqref{nsig-up} and the definitions of $T,\tau$.
\end{proof}

In a similar spirit, we have a fourth moment bound:

\begin{lemma}[$L^4$ bound]\label{l4-bound}  If $I \subset [\sigma_-, \sigma_+]$ for some fixed $0 \leq \sigma_- \leq \sigma_+ \leq 1$, then
$$ \frac{1}{X}\int_X^{2X} |S_I(x)|^4\ dx \ll_\theta X^{(1-\theta) (1-\sigma_-) A^*(\sigma_-) + 4\theta + 4\sigma_+ - 4 + o(1)}$$
\end{lemma}

\begin{proof}  Again, this estimate is essentially in \cite[\S 2]{HB-2}, but we reproduce the argument here.  With $\psi$ as before, we can bound the left-hand side by
\begin{align*}
&\ll \sum_{\sigma_1+it_1, \dots, \sigma_4+it_4 \in {\mathcal Z}(0,T): \sigma_1, \dots, \sigma_4 \in I} c_{\sigma_1+it_1} c_{\sigma_2+it_2} \overline{c}_{\sigma_3+it_3} \\
&\quad X^{\sigma_1+\sigma_2+\sigma_3+\sigma_4+i(t_1+t_2-t_3-t_4)} \hat \psi( t_1+t_2-t_3-t_4 - i(\sigma_1+\sigma_2+\sigma_3+\sigma_4))
\end{align*}
and so by using the estimates used to bound \Cref{l2-bound}, we can bound this expressin in turn by
$$ \ll X^{4\theta+4\sigma_+-4+o(1)} \sum_{\sigma_1+it_1, \dots, \sigma_4+it_4 \in {\mathcal Z}(\sigma_-,T)}
\frac{1}{(1 + |t_1+t_2-t_3-t_4|)^{10}}.$$
By a routine double counting, we can bound this expression by
$$ \ll X^{4\theta+4\sigma_+-4+o(1)} \int_\R \int_\R F(t) F(t') \frac{dt dt'}{(1+|t-t'|)^{10}}$$
where
$$ F(t) \coloneqq \# \left\{ \sigma_1+it_1, \sigma_2+it_2 \in {\mathcal Z}(\sigma_-,T) : |t_1+t_2-t| \leq \frac{1}{2} \right\}.$$
By Schur's test, this is bounded by
$$ \ll X^{4\theta+4\sigma_+-4+o(1)} \int_\R F(t)^2\ dt$$
which by expanding out the integral can be bounded by
$$ \ll X^{4\theta+4\sigma_+-4+o(1)} N^*(\sigma_-, T)$$
and the claim follows from \eqref{nsig-up} and the definitions of $T,\tau$.
\end{proof}

We can now prove \eqref{targ}.  We subdivide the interval $[0,1)$ into $J$ intervals $I$ of the form $[j/J,(j+1)/J)$ for $0 \leq j < J$.  By the pigeonhole principle (and the absence of zeroes on the line $\sigma=1$), if $|S_{[0,1]}(x)| \geq \frac{\delta}{3} \frac{X}{\tau}$, then $|S_I(x)| \geq \frac{\delta}{3J} \frac{X}{\tau}$ for one of the intervals $I$.  Thus, by the triangle inequality, it suffices to show that
\begin{equation}\label{targ-2}
 \left|\left\{ X \leq x < (1+\delta/J) X: |S_I(x)| \geq \frac{\delta}{3J} \frac{X}{\tau} \right\}\right| \ll_{\delta,\theta,J,\eps} X^{\mu+4/J+o(1)}
\end{equation}
for each such interval $I$.

Write $I = [\sigma_-,\sigma_-+1/J)$. By \Cref{right-edge}, the desired claim is true if $\sigma_- \geq 1-\eta_0$ for some sufficiently small constant $\eta_0$ (depending on $\theta$, but independent of $J$), so we may assume that $\sigma_- < 1-\eta_0$.

Suppose first that $A(\sigma_-) \leq \frac{1}{1-\theta}-\eps$.  Then we apply \Cref{linfty-bound} to conclude that
$$
\sup_{X \leq x \leq 2X}|S_I(x)|
\ll_\theta X^{- (1-\sigma_-) \eps + \theta + \frac{1}{J} + o(1)} \ll_\theta X^{-\eta_0 \eps/2 + o(1)} \frac{X}{\tau}$$
if $J$ is large enough.  Thus, for $X$ large enough, the left-hand side of \eqref{targ-2} vanishes.

It remains to consider the case when
$$ A(\sigma_-) > \frac{1}{1-\theta}-\eps.$$
By definition of $\mu$, at least one of \eqref{muth-alt}, \eqref{muth-alt-2} hold.  Suppose first that \eqref{muth-alt} holds. From \Cref{l2-bound} we have
$$ \frac{1}{X}\int_X^{2X} |S_I(x)|^2\ dx \ll_\theta X^{2\theta + \mu - 1 + 2/J + o(1)}$$
and the desired bound \eqref{targ-2} follows from the Markov inequality since $\frac{\delta}{3J} \frac{X}{\tau} =X^{2\theta+o(1)}$.
Similarly, if \eqref{muth-alt-2} holds, we have
$$ \frac{1}{X}\int_X^{2X} |S_I(x)|^4\ dx \ll_\theta X^{4\theta + \mu - 2 + 4/J + o(1)}$$
and the desired bound \eqref{targ-2} again follows from Markov's inequality as before.

\begin{remark*}
By a routine modification of the above arguments one can replace $\min(\mu_{2,\theta}, \mu_{4,\theta})$ by $\min_{1 \leq k \leq K} \mu_{2k,\theta}$ for any fixed $K$, where
$$ \mu_{2k,\sigma}(\theta) \coloneqq (1-\theta)(1-\sigma) A^{(2k)}(\sigma) + 2k\sigma-2k+1$$
and $A^{(2k)}(\sigma)$ is defined similarly to $A(\sigma)$ or $A^*(\sigma)$, but with the quantities $N(\sigma,T)$, $N^*(\sigma,T)$ replaced with the higher order additive energy
$$N^{(k)}(\sigma, T) =  \#\left\{ (\rho_1, \dots, \rho_{2k}) : \rho_j = \beta_j + i\gamma_j\in \mathcal{Z}(\sigma, T), |\gamma_1 +\dots+ \gamma_k - \gamma_{k+1} - \gamma_{2k}|\le1 \right\}.$$
The possibility of using such higher order energies in this fashion was first explored in \cite{baz-3}.  Unfortunately there are no known unconditional bounds on these energies that are not trivial consequences of existing bounds on $N(\sigma,T)$ or $N^*(\sigma,T)$.\end{remark*}

\section{Numerical calculations}\label{numerics-sec}

We now report on the bounds on $\mu(\theta)$ one obtains from \Cref{refined-bound} (or \Cref{gen-bound}) after applying various known or conjectured bounds on $A(\sigma)$.

As mentioned in the introduction, on the Riemann hypothesis (RH) we have that $\mu(\theta) \leq 1-\theta$ for $0 < \theta \leq 1/2$, and $\mu(\theta)=-\infty$ for $1/2 < \theta \leq 1$.  This appears to be the limit of what one can accomplish purely from the explicit formula.

If one only assumes the Lindel\"of hypothesis (LH), then as mentioned in the introduction, we know that $A(\sigma) \leq 2$ for $1/2 \leq \sigma \leq 3/4$ and $A(\sigma) \leq 0$ for $3/4 < \sigma < 1$.  From \eqref{muth-simp} we then have $\mu(\theta)=-\infty$ for $1/2 < \theta \leq 1$, while for $0 < \theta \leq 1/2$
 $$
   \mu(\theta) \leq \max \left( 1-\theta, \sup_{1/2 < \sigma \leq 3/4} \mu_{2,\sigma}(\theta) \right)
$$ 
and for $1/2 < \sigma \leq 3/4$ we have
$$
\mu_{2,\sigma}(\theta) \leq 2(1-\theta)(1-\sigma) + 2\sigma-1 \leq 1 - \frac{\theta}{2}
$$
and thus
$$ \mu(\theta) \leq 1 - \frac{\theta}{2}$$
for $0 < \theta \leq 1/2$.

If one just assumes the density hypothesis (DH), then we again have $\mu(\theta)=-\infty$ for $1/2 < \theta \leq 1$.  However, to get non-trivial bounds on $\mu(\theta)$ for $0 < \theta \leq 1/2$ by this method, one needs to take advantage of further zero density estimates (improving upon DH) for $\sigma$ close to one, since $\mu_{2,\sigma}(\theta)$ converges to $1$ in the limit $\sigma \to 1$.  For instance, the condition $A(\sigma) \geq \frac{1}{1-\theta}-\eps$, combined with the zero-density estimate of Pintz \cite{pintz} ensures that $\sigma \leq \frac{23}{24}$ for $\eps$ small enough, which when combined with the density hypothesis gives
$$
\mu_{2,\sigma}(\theta) \leq 2(1-\theta)(1-\sigma) + 2\sigma-1 \leq 1 - \frac{\theta}{12}
$$
and thus
$$ \mu(\theta) \leq 1 - \frac{\theta}{12}$$
for $0 < \theta \leq 1/2$.  But this bound is not optimal, due to the known improvements upon the density hypothesis for $\sigma > 25/32$, as well as bounds on the additive energy of zeroes.  The explicit evaluation of the bounds produced by \Cref{refined-bound} is complicated to state exactly, but can be performed numerically without difficulty; see \Cref{fig:conditional}.

\begin{figure}
  \centering
  \includegraphics[width=0.8\textwidth]{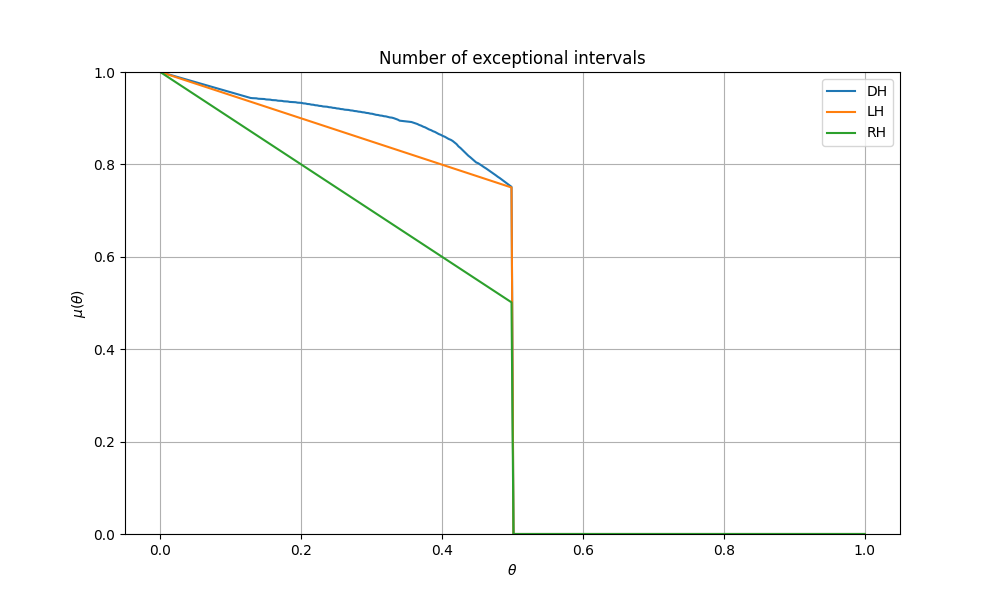}
  \caption{Upper bounds on $\mu(\theta)$ assuming the density, Lindel\"of, and Riemann hypotheses.
  }\label{fig:conditional}
  \end{figure}

The estimates in \Cref{baz-bound}(iii) can be deduced from the the following classical estimates:
\begin{itemize}
  \item The Ingham--Huxley estimate \cite{ingham}, \cite{huxley}
$$ A(\sigma) \leq \min\left( \frac{3}{2-\sigma}, \frac{3}{3\sigma-1}\right),$$
valid for all $1/2 < \sigma < 1$;
\item The Jutila density estimate \cite{jutila} $A(\sigma) \leq 2$, valid for all $\sigma > 11/14$; and
\item The Heath-Brown estimates \cite{HB-2}, which give
$$A^*(\sigma) \leq \frac{10-11\sigma}{(2-\sigma)(1-\sigma)}$$
for $1/2 \leq \sigma \leq 2/3$,
$$A^*(\sigma) \leq \frac{18-19\sigma}{(4-2\sigma)(1-\sigma)}$$
for $2/3 < \sigma \leq 3/4$, and
$$A^*(\sigma) \leq \frac{12}{4\sigma-1}$$
for $3/4 < \sigma < 1$.
\end{itemize}

For instance, in the range $\frac{1}{2} \leq \theta \leq \frac{11}{21}$, the Jutila estimate and the condition $A(\sigma) \geq \frac{1}{1-\theta}-\eps$ then restricts one to the range $\sigma \leq \frac{11}{14}$.  Here, one can use the Heath-Brown estimates and some calculation to obtain the bound
$$ \mu_{4,\sigma}(\theta) \leq \mu_{4,3/4}(\theta) = \frac{3(1-\theta)}{2},$$
giving the first part of \Cref{baz-bound}(iii); the other two cases are treated by a similar, but more complicated, analysis.

In \cite{baz-2} it was also claimed that these methods give the bounds
$$ \mu(\theta) \leq \frac{11-6\theta}{10}$$
for $\frac{1}{6} < \theta \leq \frac{121}{173}$ and
$$ \mu(\theta) \leq \frac{3(1-\theta)}{2}$$
for $\frac{121}{173} \leq \theta \leq \frac{1}{2}$.  Unfortunately, we were not able to replicate these results.  In our language, the problem is that the Jutila density estimate no longer permits one to restrict to the range $\sigma \leq \frac{11}{14}$ when $\theta < 1/2$.  In \cite{baz-2}, the corresponding issue is with  the bound in \cite[(13)]{baz-2}: if one takes for instance $\sigma = \frac{11}{14}+2\eps$ and applies \cite[(12)]{baz-2}, one needs $X^{\theta-1+\sigma} T^{2(1-\sigma)}$ to be significantly less than $X^\theta$, but as $T$ is essentially $X^{1-\theta}$, this is only true for $\theta > 1/2$.  This issue will also restrict the range of $\alpha$ for which the proofs of \cite[Theorems 1,2]{baz-2} are complete.

By a straightforward\footnote{Code for this calculation will be uploaded to \cite{antedb}.} computer calculation (discretizing $\sigma$, $\theta$ to a finite range) one can numerically compute the bound on $\mu(\theta)$ obtained from \Cref{refined-bound} combined with \Cref{zero_density_table} and \Cref{zero-density-energy-table}; see \Cref{fig:unconditional}.  Interestingly, for a significant range of $\theta$, these bounds would not be improved under the assumption of the density hypothesis.  One could in principle give exact formulae for these bounds, but they are not particularly enlightening (and will likely be superseded as additional zero density estimates are proven).  We content ourselves with two sample calculations:
\begin{itemize}
  \item If $\theta=1-1/A_0 = 17/30$ for $A_0 = 30/13$, then the condition $A(\sigma) \geq \frac{1}{1-\theta}-\eps$ (together with the Ingham and Guth--Maynard zero density estimates) places $\sigma$ arbitrarily close to $7/10$, and then by the Heath--Brown estimates we have $A^*(\theta)$ arbitrarily close to $235/39$, and then $\mu_{4,\sigma}(\theta)$ is arbitrarily close to $7/12$; thus
  $$ \mu(17/30) \leq 7/12.$$
 \item If $\theta=1-2/A_0+\Delta = 2/15+\Delta$ for $A_0=30/13$ and $\Delta > 0$ is sufficiently small, then the condition $A(\sigma) \geq \frac{1}{1-\theta}-\eps$ restricts $\sigma$ to be bounded away from $1$ (for instance one can take $\sigma \leq 23/24$).  A routine calculation then shows that
 $$ \mu_{2,\sigma}(\theta) \leq \mu_{2,7/10}(\theta) = 1 - \frac{9}{13} \Delta$$
 and thus
 $$ \mu(2/15+\Delta) \leq 1 - \frac{9}{13} \Delta.$$ 
\end{itemize}

\begin{figure}
  \centering
  \includegraphics[width=0.8\textwidth]{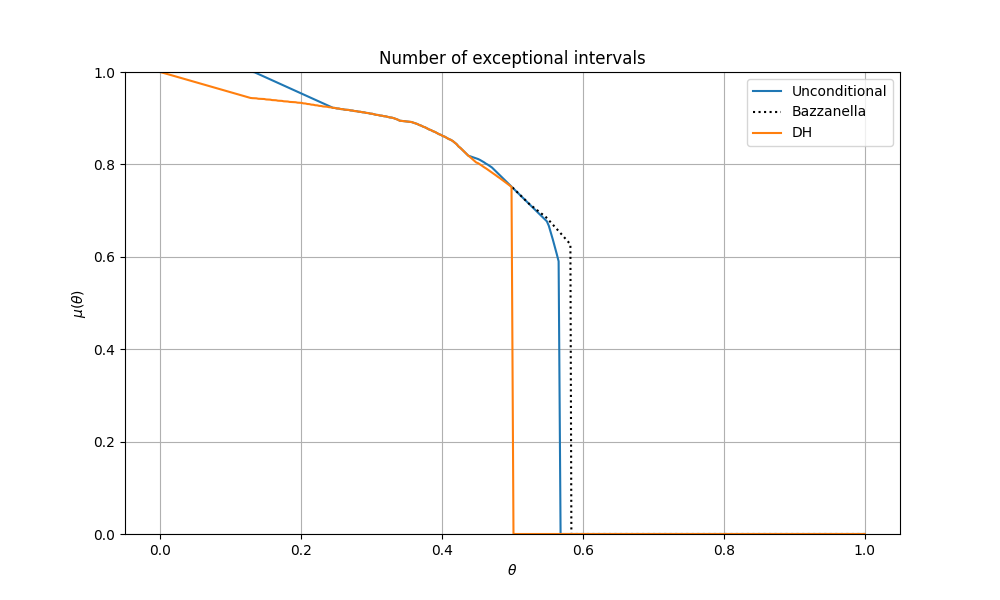}
  \caption{Best known unconditional bounds on $\mu(\theta)$, compared against the prior bounds of Bazzanella (for $\theta \geq 1/2$), as well as the stronger bounds assuming the density hypothesis.
  }\label{fig:unconditional}
  \end{figure}

\bibliographystyle{amsplain}

\end{document}